\documentclass{article}
\usepackage{amssymb,amsmath,amsthm,graphicx}
\usepackage[all,color]{xy}

\def\qed{\hfill {\hbox{${\vcenter{\vbox{               
   \hrule height 0.4pt\hbox{\vrule width 0.4pt height 6pt
   \kern5pt\vrule width 0.4pt}\hrule height 0.4pt}}}$}}}

\newtheorem{theorem}{Theorem}

\newtheorem{proposition}[theorem]{Proposition}
\newtheorem{corollary}[theorem]{Corollary}

\theoremstyle{definition}
\newtheorem{example}{Example}
\newtheorem{definition}{Definition}
\newtheorem{remark}{Remark}

\date{}

\title{\Large \textbf{Entropic Niebrzydowski Tribrackets}}

\author{Jieon Kim\footnote{Email:jieonkim7@gmail.com. Supported by Young Researchers Program through the National Research Foundation of Korea (NRF) funded by the Ministry of Education, Science and Technology (NRF-2018R1C1B6007021)} \and
Sam Nelson\footnote{Email: Sam.Nelson@cmc.edu. Partially supported by Simons Foundation Collaboration Grant 702597.}}

\begin{document}
\maketitle

\begin{abstract}
We introduce the notion of \textit{entropic Niebrzydowski tribrackets}
or just \textit{entropic tribrackets}, analogous to
\textit{entropic} (also known as \textit{abelian} or \textit{medial}) 
quandles and biquandles. We show that if $X$ is a finite entropic 
tribracket then for any tribracket $T$, the homset $\mathrm{Hom}(T,X)$
(and in particular, for any oriented link $L$, the homset 
$\mathrm{Hom}(\mathcal{T}(L),X)$) also has the structure of an entropic 
tribracket. This operation yields a product on the category of entropic
tribrackets; we compute the operation table for entropic tribrackets of small
cardinality and prove a few results. We conjecture that this structure 
can be used to distinguish links which have the same counting invariant 
with respect to a chosen entropic coloring tribracket $X$.
\end{abstract}

\parbox{5.5in}{\textsc{Keywords}: Tribrackets, Entropic tribrackets, 
Knot invariants 

\smallskip

}

\section{\large\textbf{Introduction}}\label{I}

In \cite{CN} the second listed author together with collaborator Alisa Crans 
considered \textit{abelian quandles}, also known as \textit{medial quandles} or
\textit{entropic quandles}. This class of quandles is characterized 
by the the property that for any pair of such quandles, the set of 
homomorphisms from one to the other inherits a quandle structure in a natural
way. The notion was generalized to the case of biquandles and it was shown 
that for such quandles and biquandles, the (bi)quandle structure of
the homset from the fundamental (bi)quandle of an oriented knot or link
to a finite (bi)quandle $X$ is also an invariant of knots and links.

In this paper we consider the analogous question for a more recent algebraic
structure related to oriented knots and link known as \textit{Niebrzydowski 
tribrackets} (or just \textit{tribrackets}), also known as 
\textit{knot-theoretic ternary quasigroups} \cite{N,NOO,N2,O}. 
Tribrackets are sets with a ternary operation satisfying axioms coming from the
Reidemeister moves in knot theory. In particular, finite tribrackets
define invariants of knots and links in terms of tribracket
homsets, which can be computed from diagrams. The elements of such a homset
can be represented as \textit{region colorings}, i.e. assignments of elements 
of the finite tribracket to the regions in a diagram of the knot or link
satisfying a condition at each crossing. This representation of tribracket
homsets motivates a natural tribracket-style operation on the homset itself.
We say a tribracket is \textit{entropic} if this operation satisfies the
tribracket axioms. Other related structures include \textit{tridles} \cite{Y},
\textit{quazoids} \cite{K} and \textit{biquasiles} \cite{CDN,KN,DN}.

If $X$ is a finite tribracket, then for any oriented knot or link $L$, the 
homset $\mathrm{Hom}(\mathcal{T}(L),X)$ from the fundamental tribracket of 
$L$ to $X$ is an invariant of knots and links. We conjecture that, as in 
the (bi)quandle case, there should be a finite tribracket $X$ and links $L,L'$
such that the homset tribrackets have the same cardinality but are not 
isomorphic, and we ask what is the smallest such example.

The paper is organized as follows. In Section \ref{T} we recall 
the basics of tribrackets. In Section \ref{ET} we introduce the definition
of entropic tribrackets and show that homsets from tribrackets to entropic 
tribrackets have the structure of entropic tribrackets themselves in a 
natural way. We use this fact to introduce a product $\ast$ on the category
of entropic tribrackets. In Section \ref{E} we collect some examples and 
computations, computing the operation table of the homset product for 
isomorphism classes of tribrackets of small cardinality and establishing a 
few results about the structure of these sets.  We conclude 
in Section \ref{Q} with some questions, conjectures and future directions.

\section{\large\textbf{Tribrackets}}\label{T}

We begin with a definition; see \cite{NOO,N} for more.

\begin{definition}
Let $X$ be a set. A \textit{horizontal tribracket} structure on $X$ 
is a ternary operation $X\times X\times X\to X$ denoted by $[x,y,z]$
satisfying the properties
\begin{itemize}
\item[(i)] For all $x,y,z\in X$ there are unique elements $a,b,c\in X$
satisfying
\[[a,x,y]=z,\quad [x,b,y]=z\quad \mathrm{and}\quad [x,y,c]=z\]
and
\item[(ii)] For all $x,y,z,w\in X$ we have
\[[y,[x,y,z],[x,y,w]] = [z,[x,y,z],[x,z,w]]=[w,[x,y,w],[x,z,w]].\]
\end{itemize}
\end{definition}

\begin{remark}
For every horizontal tribracket structure on a set $X$ there is a related
\textit{vertical tribracket} structure; see \cite{NOO} for more details.
In this paper we will stick to the horizontal tribracket notation.
\end{remark} 

The tribracket axioms are motivated by the Reidemeister moves using the 
region-coloring rule
\[\includegraphics{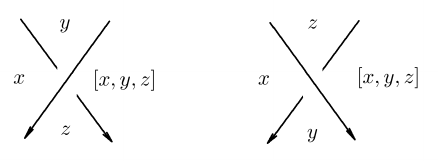}.\]
The reader can then easily verify (or check in \cite{NOO}) that the tribracket 
axioms are the condition required so that for any tribracket-colored oriented
knot or link diagram before a Reidemeister move, there is a unique
tribracket coloring of the diagram after the move which agrees with the 
original coloring outside the neighborhood of the move.

\begin{example}
A commutative ring $R$ with identity becomes a tribracket with a choice
of two units $t,s\in R^{\times}$ via the operation
\[[x,y,z]=ty+sz-tsx.\]
This tribracket structure is known as an \textit{Alexander tribracket.}
\end{example}

\begin{example}
A group $G$ is a tribracket known as a \textit{Dehn tribracket} with
\[[x,y,z]=yx^{-1}z.\]
\end{example}

\begin{example}
A tribracket structure on a finite set $X=\{1,2,\dots,n\}$ can be given
by an \textit{operation 3-tensor}, i.e. an $n$-tuple of $n\times n$ matrices.
To read such a table, we note that $[i,j,k]$ is the entry in row $j$ column
$k$ of matrix $i$. For instance, the 3-tensor
\[\left[\left[\begin{array}{rr} 1 & 2 \\ 2 & 1\end{array}\right],
\left[\begin{array}{rr} 2 & 1\\ 1 & 2\end{array}\right]\right]\]
defines a tribracket structure on $X=\{1,2\}$ and we have 
$[1,1,2]=2$.
\end{example}

\begin{example}
Every oriented knot or link $L$ has a \textit{fundamental tribracket} 
$\mathcal{T}(L)$ which can be described by a presentation with a generator
for each region in the planar complement of a diagram of $L$ and with
a relation at each crossing given by the coloring rule above.
\end{example}

\begin{definition}
A map $f:X\to Y$ between tribrackets is a \textit{tribracket
homomorphism} if for all $x,y,z\in X$ we have
\[[f(x),f(y),f(z)]=f([x,y,z]).\]
\end{definition}

\begin{theorem}
The inverse of a bijective tribracket homomorphism $f:X\to Y$
is a tribracket homomorphism $f^{-1}:Y\to X$.
\end{theorem}

\begin{proof}
We must show that $f^{-1}([x,y,z])=[f^{-1}(x),f^{-1}(y),f^{-1}(z)]$. 
We have
\begin{eqnarray*}
f^{-1}([x,y,z]) 
& = & f^{-1}([f(f^{-1}(x)),f(f^{-1}(y)),f(f^{-1}(z))]) \\
& = & f^{-1} f([f^{-1}(x),f^{-1}(y),f^{-1}(z)]) \\
& = & [f^{-1}(x),f^{-1}(y),f^{-1}(z)] 
\end{eqnarray*}
as required.
\end{proof}

\begin{definition}
A bijective tribracket homomorphism is an \textit{isomorphism}.
\end{definition}

Now let $X$ be a finite tribracket and $L$ an oriented knot or link. Then the
\textit{tribracket homset invariant} is the set
\[\mathrm{Hom}(\mathcal{T}(L),X)=\{f:\mathcal{T}(L)\to X \ 
\mathrm{homomorphism}\}\]
of tribracket homomorphisms from the fundamental tribracket of $L$
to $X$. It is an invariant of knots and links by construction since
Reidemeister moves induce isomorphisms of the fundamental tribracket.
Elements of the homset can be represented as colorings of a diagram of $L$,
with different diagrams representing the same coloring if they are related
by $X$-colored Reidemeister moves.

More precisely, a coloring of a diagram by elements of $X$ is an assignment of 
an image in $X$ to each generator of $\mathcal{T}(L)$; such an assignment
determines a tribracket homomorphism provided the crossing relations are
satisfied in $X$, and every homomorphism $f:\mathcal{T}(L)\to X$ has such
a representation. Colorings of diagrams are analogous to representing 
linear transformations as matrices whose columns are the images of basis 
elements, and Reidemeister moves are analogous to changes of basis.

In particular the cardinality of the homset is a non-negative integer-valued
invariant of knots and links known as the \textit{tribracket counting
invariant}, denoted $\Phi_X^{\mathbb{Z}}(L)=|\mathrm{Hom}(\mathcal{T}(L),X)|$.


\section{\large\textbf{Entropic Tribrackets}}\label{ET}

\begin{definition}\label{def1}
We say a tribracket is \textit{entropic} if for all $x,y,z,u,v,w,a,b,c\in X$
we have
\[[[x,y,z],[u,v,w],[a,b,c]]=[[x,u,a],[y,v,b],[z,w,c]].\]
\end{definition}

Our motivation for Definition \ref{def1} is to find the tribracket version
of the abelian property for quandles. In \cite{CN} this condition is shown to
be precisely the condition required for the homset $\mathrm{Hom}(Q(K),X)$
of quandle colorings of a diagram $D$ representing a knot $K$ to be a quandle
under the arcwise operation as shown:
\[\includegraphics{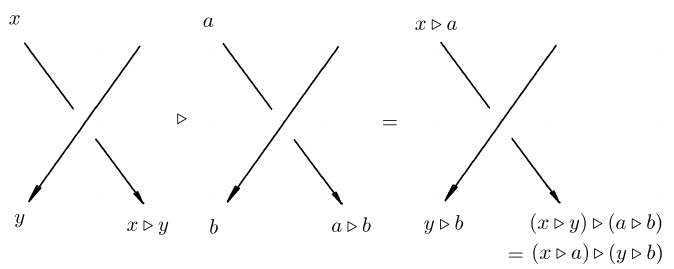}\]

Applying this idea to the tribracket case, we have
\begin{equation} \scalebox{0.9}{\includegraphics{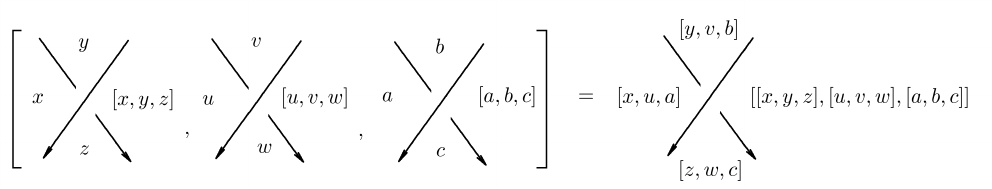}} \label{eq1}\end{equation}
which yields the condition in Definition \ref{def1}.

\begin{example}
Alexander tribrackets are entropic:
\begin{eqnarray*}
[[x,y,z],[u,v,w],[a,b,c]]  
& = & -ts[x,y,z]+t[u,v,w]+s[a,b,c] \\  
& = & -ts(-tsx+ty+sz)+t(-tsu+tv+sw)+s(-tsa+tb+sc) \\  
& = & t^2s^2x-t^2sy-ts^2z-t^2su+t^2v+tsw-ts^2a+tsb+s^2c
\end{eqnarray*}
and
\begin{eqnarray*}
[[x,u,a],[y,v,b],[z,w,c]] 
& = & -ts[x,u,a]+t[y,v,b]+s[z,w,c]\\ 
& = & -ts(-tsx+tu+sa)+t(-tsy+tv+sb)+s(-tsz+tw+sc)\\ 
& = & t^2s^2x-t^2su-ts^2a-t^2sy+t^2v+tsb-ts^2z+tsw+s^2c \\
& = & t^2s^2x-t^2sy-ts^2z-t^2su+t^2v+tsw-ts^2a+tsb+s^2c
\end{eqnarray*}
as required.
\end{example}

\begin{example}
Dehn tribrackets are not generally entropic:
\begin{eqnarray*}
[[x,y,z],[u,v,w],[a,b,c]]  
& = & [a,b,c][x,y,z]^{-1}[u,v,w]\\
& = & ca^{-1}b(zx^{-1}y)^{-1}wu^{-1}v\\
& = & ca^{-1}by^{-1}xz^{-1}wu^{-1}v
\end{eqnarray*}
while 
\begin{eqnarray*}
[[x,u,a],[y,v,b],[z,w,c]] 
& = & [z,w,c][x,u,a]^{-1}[y,v,b]\\
& = & cz^{-1}w(ax^{-1}u)^{-1}by^{-1}v\\
& = & cz^{-1}wu^{-1}xa^{-1}by^{-1}v
\end{eqnarray*}
which is not equal to $ca^{-1}by^{-1}xz^{-1}wu^{-1}v$ in general.

More precisely, for an abelian group $G$ the Dehn tribracket is entropic
but for non-abelian groups the Dehn tribracket is generally not entropic.
For example, the Dehn tribracket of $S_3$ is non-entropic.
\end{example}

\begin{proposition}
Let $T$ be a tribracket with tribracket operation $[\ ,\ ,\ ]_T$ and 
$X$ a finite entropic tribracket with tribracket operation $[\ ,\ ,\ ]_X$. 
Then the operation 
$[\ ,\ ,\ ]_H:(\mathrm{Hom}(T,X))^3\to\mathrm{Hom}(T,X)$ defined by
\[[f,g,h]_H(t)=[f(t),g(t),h(t)]_X\]
defines an entropic tribracket structure on the homset.
\end{proposition}

\begin{proof}
Let $f,g,h:T\to X$ be tribracket homomorphisms and $x,y,z\in X$. 
We must show that $[\ ,\ ,\ ]_H$ satisfies the tribracket axioms and the
entropic condition.

First let us show that $[f,g,h]_H$ is an element of the homset. Consider the
evaluation of $[f,g,h]$ at the element $[x,y,z]_T\in T$. By the definition of 
$[f,g,h]$, we have
\[[f,g,h]_H([x,y,z]_T) = [f([x,y,z]_T),g([x,y,z]_T),h([x,y,z]_T)]_X \]
and since $f,g,h$ are tribracket homomorphisms this is equal to
\[[[f(x),f(y),f(z)]_X,[g(x),g(y),g(z)]_X,[h(x),h(y),h(z)]_X]_X\]
Then the fact that $X$ is entropic implies this is equal to
\[[[f(x),g(x),h(x)]_X,[f(y),g(y),h(y)]_X,[f(z),g(z),h(z)]_X]_X\]
and be definition of $[f,g,h]$ we have
\[[[f,g,h]_H(x),[f,g,h]_H(y),[f,g,h]_H(z)]_X\]
as required.

For axiom (i), we must show that there exist unique tribracket homomorphisms 
$a,b,c\in \mathrm{Hom}(T,X)$ such that
\[[a(t),f(t),g(t)]_x=h(t),\quad [f(t),b(t),g(t)]_x=h(t)\quad 
\mathrm{and}\quad [f(t),g(t),c(t)]_x=h(t).\]
Let us consider the case $[a(t),f(t),g(t)]=h(t)$; the others are similar.
For each $t\in T$, evaluation at $t$ of $f,g,h$ determines uniquely the value 
of $a(t)$. In this way the functions $a,b,c:T\to X$ are defined and satisfy 
the conditions of axiom (i). To see that the function $a:T\to X$ thus defined 
is a tribracket homomorphism, suppose $x,y,z\in T$; we must show that 
$a([x,y,z]_T)=[a(x),a(y),a(z)]_x$. Applying the left-inverse operation 
in $X$, the reasoning in the previous paragraph yields the result.

To see that for all $f,g,h,k\in \mathrm{Hom}(T,X)$ we have
\[[g,[f,g,h]_H,[f,g,k]_H]_H = [h,[f,g,h]_H,[f,h,k]_H]_H
=[k,[f,g,k]_H,[f,h,k]_H]_H\]
we simply note that after evaluation at each element of $T$ the equations
hold in $X$ since $X$ is a tribracket, and that since this is the case for all
$t\in T$, the equation holds for $f,g,h,k$ as maps, as required.

A similar argument shows that $\mathrm{Hom}(T,X)$ is entropic, as we are done.
\end{proof}

\begin{proposition}
Let $X$ be an entropic tribracket.
An isomorphism $\phi:T\to T'$ of tribrackets induces an isomorphism
$\phi^*:\mathrm{Hom}(T',X)\to \mathrm{Hom}(T,X)$.
\end{proposition}

\begin{proof}

Let $\phi:T\to T'$ be an isomorphism of tribrackets. Then 
define $\phi^*:\mathrm{Hom}(T',X)\to \mathrm{Hom}(T,X)$ by 
\[\phi^*(f)(t)=f(\phi(t))\]
for every $f:T\to X$. Then we have
\[\phi^*([f,g,h])(t)
=[f,g,h](\phi(t))
=[f(\phi(t)),g(\phi(t)),h(\phi(t))]
=[\phi^*f(t),\phi^*g(t),\phi^*h(t)]
\]
and $\phi^*$ is a homomorphism of tribrackets. A similar argument shows
that $(\phi^{-1})^*=(\phi^{*})^{-1}:\mathrm{Hom}(T,X)\to \mathrm{Hom}(T',X)$ 
is also a homomorphism, and hence $\phi^*$ is an isomorphism.
\end{proof}

\begin{corollary}
Let $L$ be an oriented knot or link represented by a diagram $D$ and let
$X$ be an entropic tribracket. Then the homset 
$\mathrm{Hom}(\mathcal{T}(L),X)$ is
an entropic tribracket under the operation in Equation (\ref{eq1}). Moreover, 
if $L$ is ambient isotopic to $L'$, then $\mathrm{Hom}(\mathcal{T}(L),X)$
is isomorphic as a tribracket to $\mathrm{Hom}(\mathcal{T}(L'),X)$.
\end{corollary}

\section{\large\textbf{Examples and Computations}}\label{E}

In this section we will illustrate the homset construction and resulting
invariant with some examples.

\begin{example}\label{ex7}
Let $X=\{1,2\}$ have the tribracket structure given by the 3-tensor
\[\left[\left[\begin{array}{rr} 1 & 2 \\ 2 & 1\end{array}\right],
\left[\begin{array}{rr} 2 & 1\\ 1 & 2\end{array}\right]\right].\]
We can verify via computer (or, in principle, by hand) that this
tribracket structure satisfies the entropic condition by checking that
the condition
\[[[x,y,z],[u,v,w],[a,b,c]]=[[x,u,a],[y,v,b],[z,w,c]]\]
holds for all $2^9$ assignments of elements of $X$ to the variables 
$x,y,z,u,v,w,a,b,c$; for example, setting 
$(x,y,z,u,v,w,a,b,c)=(1,1,2,2,1,2,1,2,1)$ we have
\begin{eqnarray*}
{}[[x,y,z],[u,v,w],[a,b,c]] 
& = & [[1,1,2],[2,1,2],[1,2,1]] \\
& = & [2,1,2] \\
& = & 1 \\
& = & [2,2,1] \\  
& = & [[1,2,1],[1,1,2],[2,2,1]]
\end{eqnarray*}
and so forth.
\end{example}

\begin{example}\label{X4}
We then compute that the trefoil knot $3_1$ has four colorings by the 
tribracket $X$ in Example \ref{ex7} which
we can number 1 through 4, namely
\[\includegraphics{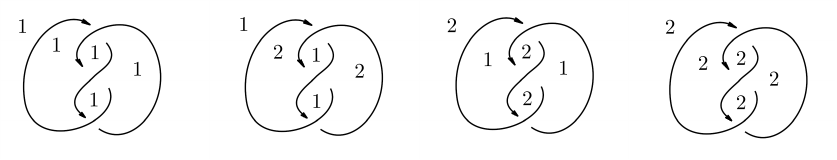}.\]
Then in the homset, we can compute that for example
\[\includegraphics{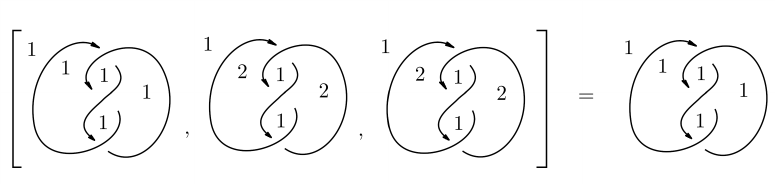}.\]
Filling in the rest of the tensor and replacing the diagrams with just the
numbers 1-4, we obtain the homset tribracket tensor
\[\left[
\left[\begin{array}{rrrr}
1 & 2 & 3 & 4 \\
2 & 1 & 4 & 3 \\
3 & 4 & 1 & 2 \\
4 & 3 & 2 & 1
\end{array}\right],
\left[\begin{array}{rrrr}
2 & 1 & 4 & 3 \\
1 & 2 & 3 & 4 \\
4 & 3 & 2 & 1 \\
3 & 4 & 1 & 2
\end{array}\right],
\left[\begin{array}{rrrr}
3 & 4 & 1 & 2 \\
4 & 3 & 2 & 1 \\
1 & 2 & 3 & 4 \\
2 & 1 & 4 & 3
\end{array}\right],
\left[\begin{array}{rrrr}
4 & 3 & 2 & 1 \\
3 & 4 & 1 & 2 \\
2 & 1 & 4 & 3 \\
1 & 2 & 3 & 4
\end{array}\right]
\right].\]
\end{example}

The homset construction defines a kind of binary product on the objects in the
category of finite entropic tribrackets by setting $X\ast Y=\mathrm{Hom}(X,Y)$.
Using \texttt{python} code, we computed the set of entropic tribrackets
with up to four elements up to isomorphism. In the table below the subscript
gives the cardinality of the tribracket and the superscript is the numbering
in the table; we use $T_0^1$ for the empty tribracket and $T_1^1$ for the
tribracket of one element. 

\[\begin{array}{r|l}
\mathrm{Tribracket} & \mathrm{Operation\ Tensor} \\ \hline & \\
T_2^1 &  
\left[\left[\begin{array}{rr} 1 & 2 \\ 2 & 1 \end{array}\right],
      \left[\begin{array}{rr} 2 & 1 \\ 1 & 2\end{array}\right]\right] \\ & \\
T_2^2 &  
\left[\left[\begin{array}{rr} 2 & 1 \\ 1 & 2 \end{array}\right],
      \left[\begin{array}{rr} 1 & 2 \\ 2 & 1\end{array}\right]\right] \\ & \\ \hline & \\
T_3^1 &  
\left[\left[\begin{array}{rrr} 1 & 2 & 3 \\ 2 & 3 & 1 \\ 3 & 1 & 2 \end{array}\right],
      \left[\begin{array}{rrr} 3 & 1 & 2 \\ 1 & 2 & 3 \\ 2 & 3 & 1 \end{array}\right],
      \left[\begin{array}{rrr} 2 & 3 & 1 \\ 3 & 1 & 2 \\ 1 & 2 & 3 \end{array}\right]\right] \\ & \\
T_3^2 &  
\left[\left[\begin{array}{rrr} 1 & 2 & 3 \\ 3 & 1 & 2 \\ 2 & 3 & 1 \end{array}\right],
      \left[\begin{array}{rrr} 2 & 3 & 1 \\ 1 & 2 & 3 \\ 3 & 1 & 2 \end{array}\right],
      \left[\begin{array}{rrr} 3 & 1 & 2 \\ 2 & 3 & 1 \\ 1 & 2 & 3 \end{array}\right]\right] \\ & \\
T_3^3 &  
\left[\left[\begin{array}{rrr} 1 & 3 & 2 \\ 2 & 1 & 3 \\ 3 & 2 & 1 \end{array}\right],
      \left[\begin{array}{rrr} 2 & 1 & 3 \\ 3 & 2 & 1 \\ 1 & 3 & 2 \end{array}\right],
      \left[\begin{array}{rrr} 3 & 2 & 1 \\ 1 & 3 & 2 \\ 2 & 1 & 3 \end{array}\right]\right] \\ & \\
T_3^4 &  
\left[\left[\begin{array}{rrr} 1 & 3 & 2 \\ 3 & 2 & 1 \\ 2 & 1 & 3 \end{array}\right],
      \left[\begin{array}{rrr} 3 & 2 & 1 \\ 2 & 1 & 3 \\ 1 & 3 & 2 \end{array}\right],
      \left[\begin{array}{rrr} 2 & 1 & 3 \\ 1 & 3 & 2 \\ 3 & 2 & 1 \end{array}\right]\right] \\ & \\
T_3^5 &  
\left[\left[\begin{array}{rrr} 2 & 1 & 3 \\ 3 & 2 & 1 \\ 1 & 3 & 2 \end{array}\right],
      \left[\begin{array}{rrr} 3 & 2 & 1 \\ 1 & 3 & 2 \\ 2 & 1 & 3 \end{array}\right],
      \left[\begin{array}{rrr} 1 & 3 & 2 \\ 2 & 1 & 3 \\ 3 & 2 & 1 \end{array}\right]\right] \\ & \\
T_3^6 &  
\left[\left[\begin{array}{rrr} 2 & 3 & 1 \\ 1 & 2 & 3 \\ 3 & 1 & 2 \end{array}\right],
      \left[\begin{array}{rrr} 3 & 1 & 2 \\ 2 & 3 & 1 \\ 1 & 2 & 3 \end{array}\right],
      \left[\begin{array}{rrr} 1 & 2 & 3 \\ 3 & 1 & 2 \\ 2 & 3 & 1 \end{array}\right]\right] \\ & \\
T_3^7 &  
\left[\left[\begin{array}{rrr} 2 & 3 & 1 \\ 3 & 1 & 2 \\ 1 & 2 & 3 \end{array}\right],
      \left[\begin{array}{rrr} 1 & 2 & 3 \\ 2 & 3 & 1 \\ 3 & 1 & 2 \end{array}\right],
      \left[\begin{array}{rrr} 3 & 1 & 2 \\ 1 & 2 & 3 \\ 2 & 3 & 1 \end{array}\right]\right] \\ 
\end{array}\]

\begin{remark}
We note that the empty set is an entropic tribracket, which we 
denote as $T_0^1$, that the homset from any tribracket to the empty
tribracket is the empty tribracket, and that the homset from $T_0^1$ to 
any nonempty entropic tribracket is the one-element tribracket $T_1^1$ 
whose single element is the empty map.
\end{remark}

Then the product table for the homset product $X\ast Y=\mathrm{Hom}(X,Y)$ 
for entropic tribrackets of small cardinality is 
\[
\begin{array}{r|rrrrrrrrrrr}
\ast & T_0^1 & T_1^1 & T_2^1 & T_2^2 & T_3^1 & T_3^2 & T_3^3 & T_3^4 & T_3^5 & T_3^6 & T_3^7 \\ \hline 
T_0^1 & T_0^1 & T_1^1 & T_1^1 & T_1^1 & T_1^1 & T_1^1 & T_1^1 & T_1^1 & T_1^1 & T_1^1 & T_1^1 \\
T_1^1 & T_0^1 & T_1^1 & T_2^1 & T_0^1 & T_3^1 & T_3^2 & T_3^3 & T_1^1 & T_0^1 & T_0^1 & T_0^1 \\
T_2^1 & T_0^1 & T_1^1 & T_4^1 & T_0^1 & T_3^1 & T_3^2 & T_3^3 & T_1^1 & T_0^1 & T_0^1 & T_0^1 \\
T_2^2 & T_0^1 & T_1^1 & T_2^1 & T_2^2 & T_3^1 & T_3^2 & T_3^3 & T_1^1 & T_0^1 & T_0^1 & T_0^1 \\
T_3^1 & T_0^1 & T_1^1 & T_2^1 & T_0^1 & \ast & T_3^2 & T_3^3 & T_1^1 & T_0^1 & T_0^1 & T_0^1 \\
T_3^2 & T_0^1 & T_1^1 & T_2^1 & T_0^1 & T_3^1 & \ast & T_3^3 & T_1^1 & T_0^1 & T_0^1 & T_0^1 \\
T_3^3 & T_0^1 & T_1^1 & T_2^1 & T_0^1 & T_3^1 & T_3^2 & \ast & T_1^1 & T_0^1 & T_0^1 & T_0^1 \\
T_3^4 & T_0^1 & T_1^1 & T_2^1 & T_0^1 & T_3^1 & T_3^2 & T_3^3 & T_3^4 & T_0^1 & T_0^1 & T_0^1 \\
T_3^5 & T_0^1 & T_1^1 & T_2^1 & T_0^1 & T_3^1 & T_3^2 & T_3^3 & T_1^1 & T_3^5 & T_0^1 & T_0^1 \\
T_3^6 & T_0^1 & T_1^1 & T_2^1 & T_0^1 & T_3^1 & T_3^2 & T_3^3 & T_1^1 & T_0^1 & T_3^6 & T_0^1 \\
T_3^7 & T_0^1 & T_1^1 & T_2^1 & T_0^1 & T_3^1 & T_3^2 & T_3^3 & T_1^1 & T_0^1 & T_0^1 & T_3^7 \\
\end{array}
\]
where $T_4^1$ is the tribracket structure on $\{1,2,3,4\}$ in Example 
\ref{X4} and the tribrackets marked $\ast$ are nonisomorphic 9-element 
tribrackets.

Let us define an element $x\in X$ to be an \textit{idempotent element} if
$[x,x,x]=x$ and define $\mathrm{Idem}(X)$ to be the subtribracket of $X$ 
generated by idempotent elements. We note that $\mathrm{Idem}(X)$ may be
all of $X$, may be empty, or may be a proper subset of $X$. For example,
we have $\mathrm{Idem}(T_2^1)=T_2^1$ and $\mathrm{Idem}(T_2^2)=T_0^1$, while
the Alexander tribracket on $\mathbb{Z}_8$ with $s=t=7$ has a four-element
idempotent subtribracket.
 
We have the following observation:
\begin{proposition}
Let $X$ and $Y$ be entropic tribrackets and fix an element $y\in Y$. 
Then a constant map $\phi:X\to Y$ defined by $\phi(x)=y$ for all $x\in X$
is a tribracket homomorphism iff $y$ is an idempotent element of $Y$.
\end{proposition}

\begin{proof}
Suppose $y=[y,y,y]$, then we have
\[\phi([x_1,x_2,x_3])=y=[y,y,y]=[\phi(x_1),\phi(x_2),\phi(x_3)]\]
for all $x_1,x_2,x_3\in X$ and $\phi$ is a tribracket homomorphism. 
Conversely if $\phi(x)=y$ for all $x\in X$ is a tribracket homomorphism 
then we have 
\[y=\phi([x_1,x_2,x_3])=[\phi(x_1),\phi(x_2),\phi(x_3)]=[y,y,y].\]
\end{proof}

\begin{corollary}
For any entropic tribrackets $X$ and $Y$, the homset tribracket 
$\mathrm{Hom}(X,Y)$ contains a copy of $\mathrm{Idem}(Y)$.
\end{corollary}

\begin{remark}
We note that for much of the table, indeed for nearly all of the cases 
$X\ast Y$ we computed where $X\ne Y$, the homset tribracket is simply 
$\mathrm{Idem}(Y)$.
\end{remark} 

\begin{proposition}\label{prop1}
Let $X$ be an Alexander tribracket over a field. Then either 
$\mathrm{Idem}(X)=T_1^1$ or $\mathrm{Idem}(X)=X$.
\end{proposition}

\begin{proof}
We first note that $[0,0,0]=st0-s0-t0=0$ so in every Alexander tribracket,
the zero element is idempotent. More generally, the idempotence equation
\[[x,x,x]=stx-sx-tx=(st-s-t)x=x\]
is satisfied for nonzero $x$ iff $st-s-t=1$ independently of the value of 
$x$, so for Alexander tribrackets over a field either the only idempotent 
element is zero (which generates $T_1^1$) or is all of $X$.
\end{proof}

For Alexander tribrackets over rings with zero divisors, the condition
for idempotence, $(st-s-t)x=x$ may depend on $x$; for example, in the 
Alexander tribracket structure on $X=\mathbb{Z}_6$ with $s=t=5$, we have
$\mathrm{Idem}(X)=T_2^1$ generated by $\{0,3\}\in X$, since 
\[(st-s-t)x=(25-5-5)3=15(3)=45=3.\]
Hence, the statement of Proposition \ref{prop1} becomes false if we drop the 
``over a field'' condition.

Since the homset tribracket $\mathrm{Hom}(\mathcal{T}(K),X)$  is a knot 
invariant for any finite 
entropic tribracket $X$, it follows that any invariant of
tribrackets applied to the homset tribracket then gives us a new knot invariant.

\begin{definition}
Let $X$ be a tribracket. The number of elements $t\in T$ such that $[t,t,t]=t$
is called the \textit{idempotent number} of $T$.
\end{definition}

It is clear that the idempotent number of a tribracket is not changed by 
isomorphism and hence is an invariant of tribrackets. Then we have

\begin{proposition}
The idempotent number of a knot homset tribracket is an integer-valued 
invariant of knots and links.
\end{proposition}

We conclude this section with another example.
\begin{example}
The six-variable \textit{tribracket polynomial} defined in
\cite{NN} applied to the homset tribracket defines a polynomial invariant 
of oriented knots and links. The exponents of the variables in this
polynomial count the number of elements $y$ satisfying equations like
$[x,y,y]=x$ and $[x,y,y]=y$ for each $x\in X$ analogous to the exponents
in the quandle polynomial defined in \cite{N1}.
\end{example}

\section{\large\textbf{Questions}}\label{Q}

For tribrackets of small cardinality, our computations show that most finite
tribrackets are entropic; indeed, all tribrackets with up to 5 elements
are entropic. We conjecture that this property is analogous to 
abelian-ness for groups or alternating-ness for knots in that the apparent 
dominance of entropic tribrackets is an artifact of the small cardinalities 
we are able to easily access computationally. What is the asymptotic
ratio of non-entropic to entropic tribrackets as cardinality grows without 
bound?

Our \texttt{python} computations show that for finite tribrackets of small
cardinality (e.g., up to 4) and knots and links of small crossing number 
(e.g., knots up to eight crossings and links up to 7), these homset 
tribrackets are isomorphic when their cardinalities agree. For example, our
computations show that all prime classical knots with up to eight crossings
have the same homset tribracket with respect to $T_3^1$, namely 
$T_3^1\ast T_3^1$.

We conjecture that this triviality is an artifact of the small cardinalities
accessible by our current computational methods and that
for large enough finite entropic tribrackets, the isomorphism type of the
homset invariant becomes nontrivial. We ask, what is the smallest example of
a finite entropic tribracket $X$ and two oriented links $L,L'$ such that
\[|\mathrm{Hom}(\mathcal{T}(L),X)|=|\mathrm{Hom}(\mathcal{T}(L'),X)|\]
and
\[\mathrm{Hom}(\mathcal{T}(L),X)\not\cong\mathrm{Hom}(\mathcal{T}(L'),X)?\]

\bibliography{jk-sn4}{}
\bibliographystyle{abbrv}

\bigskip

\noindent
\textsc{Department of Mathematics \\ 
Pusan National University \\ 
Busan 46242, Republic of Korea
}

\medskip

\noindent
\textsc{Department of Mathematical Sciences \\
Claremont McKenna College \\
850 Columbia Ave. \\
Claremont, CA 91711}

\end{document}